\documentclass[12pt]{amsart}

\usepackage{amssymb}
\usepackage{verbatim}
\usepackage[toc,page]{appendix}
\usepackage{mathrsfs}

\newtheorem{thm}{Theorem}[section]

\newtheorem{lem}[thm]{Lemma}

\newtheorem{rem}[thm]{Remark}

\theoremstyle{definition}

\theoremstyle{remark}

\numberwithin{equation}{section}

\newcommand{\fr}{\kappa}

\newcommand{\Mod}[1]{\ (\textup{mod}\ #1)}

\DeclareMathOperator{\res}{res}
\providecommand{\sym}{\operatorname{sym}}

\begin{document}

\title[]{Bounds for a spectral exponential sum}

\author{Olga  Balkanova}
\address{Department of Mathematical Sciences, Chalmers University of Technology and University of Gothenburg, Chalmers tv\"{a}rgata 3, 412 58 Gothenburg, Sweden}
\email{olgabalkanova@gmail.com}

\author{Dmitry  Frolenkov}
\address{National Research University Higher School of Economics, Moscow, Russia and Steklov Mathematical Institute of Russian Academy of Sciences,  8 Gubkina st., Moscow, 119991, Russia}
\email{frolenkov@mi.ras.ru}

\begin{abstract}
We prove new upper bounds for a spectral exponential sum by refining the process by which one evaluates mean values of $L$-functions multiplied by an oscillating function.
In particular, we introduce a method which is capable of taking into consideration the oscillatory behaviour of the function.
This gives an improvement of the result of Luo and Sarnak when $T\geq X^{1/6+2\theta/3}$. Furthermore, this proves the conjecture of Petridis and Risager in certain ranges.

\end{abstract}


\maketitle

\tableofcontents


\section{Introduction}
This paper presents a new approach to studying the spectral exponential sum
\begin{equation}\label{spec.exp.sum}
S(T,X)=\sum_{t_j\le T}X^{it_j},
\end{equation}
where $\kappa_j=1/4+t_{j}^{2}$ are the eigenvalues of the hyperbolic Laplacian for $PSL_2(\mathbf{Z})$.
The sum \eqref{spec.exp.sum} is attracting considerable interest because it is closely related to some classical problems, including prime geodesic theorem and hyperbolic lattice counting problem.
See, for example, \cite{CG}, \cite{IwPG}, \cite{LuoSarnakPG}, \cite{PetRis}, \cite{PetRisLaak}.

The prime geodesic theorem is concerned with the asymptotic behaviour of
\begin{equation*}
\pi_{\Gamma}(X)=\#\{\{P\}:NP\leq X\},
\end{equation*}
where $\{P\}$ denotes a primitive hyperbolic class in  $PSL_2(\mathbf{Z})$ and $NP$ denotes its norm.

By partial summation, the problem can be formulated in terms of  $$\Psi_{\Gamma}(X)=\sum_{NP\leq X}\Lambda(P),$$ where the sum is over all  hyperbolic classes and $\Lambda(P)=\log{NP_0}$ if $\{P\}$ is a power of the primitive hyperbolic class $\{P_0\}$.

Iwaniec \cite[Lemma 1]{IwPG} proved for $1\le T\le X^{1/2}\log^{-2}X$ the following explicit formula
\begin{equation}\label{PrimeGeodesic to spec.sum}
\Psi_{\Gamma}(X)=X+X^{1/2}\sum_{|t_j|\le T}\frac{X^{it_j}}{1/2+it_j}+O\left(\frac{X}{T}\log^2X\right),
\end{equation}
which is closely related to the spectral exponential sum \eqref{spec.exp.sum}.

The trivial bound on \eqref{spec.exp.sum}, namely $$S(T,X)\ll T^{2},$$
follows from Weyl's law and yields
\begin{equation}\label{eq:prgeotrivb}
\Psi_{\Gamma}(X)=X+O(X^{3/4}\log{X}).
\end{equation}

In order to improve \eqref{eq:prgeotrivb} it is required to exploit cancellation in sum \eqref{spec.exp.sum}.
There are several non-trivial estimates in the literature. The first  estimate
\begin{equation}\label{Iwaniec bound}
S(T,X)\ll TX^{11/48+\epsilon}
\end{equation}
was obtained by Iwaniec in \cite{IwPG}. As a consequence, he showed that $$\Psi_{\Gamma}(X)=X+O(X^{35/48+\epsilon}).$$
Luo and Sarnak  \cite{LuoSarnakPG} proved another estimate
\begin{equation}\label{LuoSarnak bound}
S(T,X)\ll T^{5/4}X^{1/8}\log^2T
\end{equation}
and deduced from this that $$\Psi_{\Gamma}(X)=X+O(X^{7/10+\epsilon}).$$

In \cite[Eq. (7.1)]{Cai} Cai showed that
\begin{equation}\label{eq:caises}
S(T,X)\ll (TX)^{\epsilon}(T^{2/5}X^{11/30}+T^{3/2}),
\end{equation}
and consequently,
$$\Psi_{\Gamma}(X)=X+O(X^{71/102+\epsilon}).$$
Finally, the prime geodesic theorem in the strongest presently known form $$\Psi_{\Gamma}(X)=X+O(X^{2/3+\theta/6}) \text{ with } \theta=1/6+\epsilon,$$ was proved by Soundararajan and Young in \cite{SY}
by combining Luo-Sarnak's bound \eqref{LuoSarnak bound} and a new approach without recourse to the spectral exponential sum.

Petridis and Risager  \cite[Conj.  2.2]{PetRisLaak} conjectured square root cancellation in sum \eqref{spec.exp.sum}, namely
\begin{equation}\label{PetrisdisRisager conj}
S(T,X)\ll T(TX)^{\epsilon}.
\end{equation}
Furthermore, they showed that estimate \eqref{PetrisdisRisager conj} yields not only the best possible error term $O(X^{1/2+\epsilon})$  in the prime geodesic theorem, but also the best error term on average for the hyperbolic lattice problem. See \cite{PetRisLaak}  for more details.

In the appendix of \cite{PetRisLaak}, Laaksonen proved that the conjecture of Petridis and Risager is true for a fixed $X$ as $T\to\infty.$
Moreover, for
\begin{equation}\label{fr(X) def}
\fr(X):=\left\|X^{1/2}+X^{-1/2}\right\|, 
\end{equation}
 where  $\|x\|$  is the distance from $x$ to the nearest integer,
Laaksonen mentioned in \cite[Experimental Observation 2]{PetRisLaak} that $S(T,X)$ has a peak when $\fr(X)=0.$

The quantity $\fr(X)$ appears also in our result. Furthermore, our estimates depend on the parameter $\theta$, which denotes the subconvexity exponent in the conductor aspect for Dirichlet $L$-functions of real primitive characters. The current record is $\theta=1/6+\epsilon$ due to Conrey and Iwaniec \cite{CI}.


\begin{thm}\label{thm:spec.exp.sum new bound}
The following estimates hold:
\begin{equation}\label{spec.exp.sum new bound}
S(T,X)\ll \max\left(
X^{1/4+\theta/2}T^{1/2},
X^{\theta/2}T
\right)\log^{3}T;
\end{equation}
\begin{equation}\label{spec.exp.sum new bound2}
S(T,X)\ll T\log^{2}T\quad\hbox{if}\quad T>\frac{X^{1/2+7\theta/6}}{\fr(X)}.
\end{equation}
\end{thm}

Theorem \ref{thm:spec.exp.sum new bound} shows that the conjecture of Petridis and Risager is true for $T>X^{1/2}$ under the Lindel\"{o}f conjecture, and  for $T>X^{1/2+7\theta/6} /\fr(X)$ unconditionally.

Furthermore, estimate \eqref{spec.exp.sum new bound} improves \eqref{LuoSarnak bound} when $T\geq X^{1/6+2\theta/3+\epsilon}$. Consequently, this result has an application to the prime geodesic theorem, as we now show.
Using summation by parts, \eqref{LuoSarnak bound} and  \eqref{spec.exp.sum new bound} we obtain
\begin{equation}\label{PrimeGeodesic bound}
\Psi_{\Gamma}(X)=X+O\left(X^{2/3+\theta/6}\log^3X\right).
\end{equation}
The error term in asymptotic formula \eqref{PrimeGeodesic bound} coincides with the one proved by Soundararajan and Young. Note that using \eqref{eq:caises} instead of \eqref{LuoSarnak bound} would not lead to further improvement.

The proof of Theorem \ref{thm:spec.exp.sum new bound} is based on the upper bound for the mean value of Maa{\ss} symmetric square $L$-functions on the critical line multiplied by $X^{it_j}$.

\begin{thm}\label{thm:momentoscill} Let $X\gg 1$ then for $s=1/2+ir, |r|\ll T^{\epsilon}$ one has
\begin{multline}\label{symsquare estimate0}
\sum_{j}\alpha_jX^{it_j}\exp(-t_j/T)L(\sym^2 u_{j},s)
\ll(1+|r|)^AT\log^3T+\\
\frac{(1+|r|)^AX^{\theta}}{X^{1/4}}\left(
X^{1/2}\min\left(T,\frac{X^{1/2}}{\fr(X)}\right)^{1/2}+
\min\left(T,\frac{X^{1/2}}{\fr(X)}\right)^{3/2}\right).
\end{multline}
\end{thm}


The paper is organised as follows. In Section \ref{sec:introduction} we introduce the  required notation. In Section \ref{sec:initial steps}  initial steps of the proof are described following the paper of Iwaniec \cite{IwPG}.
Sections \ref{sec:exactformula}, \ref{sec:sigma1}, \ref{sec:sigma2} are devoted to proving an exact formula for the first moment of Maa{\ss} symmetric square $L$-functions multiplied by an oscillating function.
In Section \ref{sec:specialfunctions} we analyse special functions that appear in the exact formula. Consequently, we prove Theorem \ref{thm:momentoscill} in Section \ref{sec:estimatesmoment}.
The main result, namely Theorem \ref{thm:spec.exp.sum new bound}, is proved in Section \ref{sec:estimatessum}.


\section{Notation}\label{sec:introduction}

For a complex number $ v $ let
\begin{equation*}
\tau_v(n)=\sum_{n_1n_2=n}\left( \frac{n_1}{n_2}\right)^v=n^{-v}\sigma_{2v}(n),
\end{equation*}
where
\begin{equation*}\sigma_v(n)=\sum_{d|n}d^v.
\end{equation*}

Note that
\begin{equation}\label{series with tau}
\sum_{n=1}^{\infty}\frac{\tau_{ir}(n^2)}{n^s}=\frac{\zeta(s)\zeta(s+2ir)\zeta(s-2ir)}{\zeta(2s)},
\end{equation}
where $\zeta(s)$ is the Riemann zeta function.

Define the Mellin transform of  $f(x)$ as follows
\begin{equation*}
\tilde{f}(s)=\int_0^{\infty}f(x)x^{s-1}dx.
\end{equation*}

 The  Kloosterman sum is given by
\begin{equation*}
S(n,m;c)=\sum_{\substack{a\pmod{c}\\ (a,c)=1}}e\left( \frac{an+a^*m}{c}\right), \quad aa^*\equiv 1\pmod{c},
\end{equation*}
where $e(x)=exp(2\pi ix)$.
According to the Weil bound (see \cite{W})
\begin{equation}\label{Weilbound}
|S(m,n;c)|\leq \tau_0(c)\sqrt{(m,n,c)}\sqrt{c}.
\end{equation}

The Lerch zeta function
\begin{equation}
\zeta(\alpha,\beta,s)=\sum_{n+\alpha>0}\frac{e(n\beta)}{(n+\alpha)^s}
\end{equation}
satisfies the functional equation (see \cite{Ler})
\begin{equation}\label{LerchFE}
\zeta(\alpha,0,s)=\frac{\Gamma(1-s)}{(2\pi)^{1-s}}\biggl(
ie\left(-\frac{s}{4}\right)\zeta(0,-\alpha,1-s)-ie\left(\frac{s}{4}\right)\zeta(0,\alpha,1-s)\biggr).
\end{equation}

Introduce the generalized Dirichlet $L$-function (see \cite{B,Z} for details)
\begin{equation}\label{Lbyk}
\mathscr{L}_{n}(s)=\frac{\zeta(2s)}{\zeta(s)}\sum_{q=1}^{\infty}\frac{\rho_q(n)}{q^{s}}=\sum_{q=1}^{\infty}\frac{\lambda_q(n)}{q^{s}},
\end{equation}
where
\begin{equation}
\rho_q(n):=\#\{x\Mod{2q}:x^2\equiv n\Mod{4q}\},
\end{equation}
\begin{equation}
\lambda_q(n):=\sum_{q_{1}^{2}q_2q_3=q}\mu(q_2)\rho_{q_3}(n).
\end{equation}
For $n\neq 0$ one has
\begin{equation}\label{eq:subconvexity}
\mathscr{L}_n(1/2+it)\ll n^{\theta}(1+|t|)^{A},
\end{equation}
where $\theta$ and $A$ are subconvexity exponents for Dirichlet $L$-functions of real primitive characters. The best known result in the $q$ aspect, namely $\theta=1/6+\epsilon$, was first obtained by Conrey and Iwaniec \cite{CI}.
Young \cite{Y} proved the hybrid subconvexity bound with $A=\theta=1/6+\epsilon$.

Denote by $\{u_j\}$  the orthonormal basis of the space of Maa{\ss} cusp forms consisting of common eigenfunctions of all Hecke operators and the hyperbolic Laplacian. Let
$\{\lambda_{j}(n)\}$ be the eigenvalues of Hecke operators acting on $u_{j}$. Let $\kappa_{j}=1/4+t_{j}^2$ be the eigenvalues of the hyperbolic Laplacian acting on $u_{j}$.
Elements of the basis have a Fourier expansion of the following form
\begin{equation*}
u_{j}(x+iy)=\sqrt{y}\sum_{n\neq 0}\rho_{j}(n)K_{it_j}(2\pi|n|y)e(nx),
\end{equation*}
where $K_{\alpha}(x)$ is the $K$-Bessel function and
\begin{equation*}
\rho_{j}(n)=\rho_{j}(1)\lambda_{j}(n).
\end{equation*}
Note that for $n,m \geq 1$
\begin{equation}\label{eq:multipFourcoeff2}
\lambda_{j}(n)\lambda_{j}(m)=\sum_{d|(m,n)}\lambda_{j}\left( \frac{nm}{d^2}\right).
\end{equation}

Introduce the normalizing coefficient
\begin{equation}\label{alphaj}
\alpha_{j}:=\frac{|\rho_{j}(1)|^2}{\cosh{\pi t_j}}.
\end{equation}

For $\Re{s}>1$, the symmetric square $L$-function is given by
\begin{equation}
L(\sym^2 u_{j},s):=\zeta(2s)\sum_{n=1}^{\infty}\frac{\lambda_{j}(n^2)}{n^s}.
\end{equation}

For $\Re{s}>1$ define the Rankin zeta-function
\begin{equation}\label{Rankin zeta}
L(u_{j}\otimes u_{j},s):=\sum_{n=1}^{\infty}\frac{|\rho_{j}(n)|^2}{n^s}.
\end{equation}

We will use the following properties (see \cite[p. 216]{Luo} and \cite[proof of Lemma 8]{IwPG})
\begin{equation}\label{residue of Rankin zeta}
\res_{s=1} L(u_{j}\otimes u_{j},s)=\frac{2}{\zeta(2)}\cosh(\pi t_j),
\end{equation}
\begin{equation}\label{Rankin to symsquare}
L(u_{j}\otimes u_{j},s)=|\rho_{j}(1)|^2\frac{\zeta(s)}{\zeta(2s)}L(\sym^2 u_{j},s).
\end{equation}
Let $\varphi(x)$ be a smooth function on $[0,\infty)$ such that
\begin{equation*}
\varphi(0)=0,\quad \varphi^{(j)}(x)\ll(1+x)^{-2-\epsilon}, \quad j=0,1,2.
\end{equation*}
Let $J_{\nu}(x)$ be the $J$-Bessel function. Define the series of integral transforms
\begin{equation}\label{phi0-Def}
\varphi_0=\frac{1}{2\pi}\int_0^{\infty}J_0(y)\varphi(y)dy,
\end{equation}

\begin{equation}\label{phiB-Def}
\varphi_B(x)=\int_0^{1}\int_0^{\infty}\xi xJ_0(\xi x)J_0(\xi y)\varphi(y)dyd\xi,
\end{equation}

\begin{equation}\label{phiH-Def}
\varphi_H(x)=\int_1^{\infty}\int_0^{\infty}\xi xJ_0(\xi x)J_0(\xi y)\varphi(y)dyd\xi,
\end{equation}

\begin{equation}\label{phiHat-Def}
\hat{\varphi}(t)=\frac{\pi i}{2\sinh(\pi t)}\int_0^{\infty}(J_{2it}(x)-J_{-2it}(x))\varphi(x)\frac{dx}{x},
\end{equation}

\begin{equation}\label{phiCheck-Def}
\check{\varphi}(l)=\int_0^{\infty}J_l(y)\varphi(y)\frac{dy}{y}.
\end{equation}

The following decomposition holds

\begin{equation}\label{phi=phiB+phiH}
\varphi(x)=\varphi_H(x)+\varphi_B(x),
\end{equation}
where
\begin{equation}\label{phiB series Def}
\varphi_B(x)=\sum_{l\equiv1(2)}2l\check{\varphi}(l)J_l(x).
\end{equation}
\begin{lem}(Kuznetsov trace formula)
For all $m,n \geq 1$
\begin{multline}\label{eq:KuzTrForm}
\sum_{j=1}^{\infty}\alpha_j\lambda_j(m)\lambda_j(n)\hat{\varphi}(t_j)+
\frac{2}{\pi}\int_{0}^{\infty}
\frac{\tau_{ir}(m)\tau_{ir}(n)}{|\zeta(1+2ir)|^2}\hat{\varphi}(r)dr=\\
\delta(m,n)\varphi_0+
\sum_{q=1}^{\infty}\frac{S(m,n;q)}{q}\varphi_H\left( \frac{4\pi \sqrt{mn}}{q}\right).
\end{multline}
\end{lem}
\begin{proof}
See \cite{Iwbook} or \cite{Kuz} .
\end{proof}
Following \cite{DesIw} and \cite{LuoSarnakPG}\footnote{Note that there is a typo (the imaginary unit $i$ is placed in the denominator instead of the numerator) in \cite[p.68 line -1]{DesIw} and \cite[p.233 line -3]{LuoSarnakPG} in the definition of $\hat{\varphi}(t)$. See \eqref{phiHat-Def} for the corrected version. This explains the sign change in the definition of $\varphi(x)$.} let
\begin{equation}\label{phi def}
\varphi(x)=\frac{\sinh\beta}{\pi}x\exp(ix\cosh\beta)
\end{equation}
with
\begin{equation}\label{beta def}
\beta=\frac{1}{2}\log X+\frac{i}{2T}.
\end{equation}
It is useful to introduce the following  notation
\begin{equation}\label{c def}
c:=-i\cosh\beta=a-ib,
\end{equation}
\begin{equation}\label{a,b def}
\begin{cases}
a:=\sinh(\log\sqrt{X})\sin((2T)^{-1});\\
b:=\cosh(\log\sqrt{X})\cos((2T)^{-1}).
\end{cases}
\end{equation}
Note that
\begin{equation}\label{argc}
\arg{c}=-\pi/2+\gamma,\quad 0<\gamma\ll T^{-1}.
\end{equation}

\begin{lem}
The following holds
\begin{equation}\label{phihat expression}
\hat{\varphi}(t)=\frac{\sinh(\pi t+2i\beta t)}{\sinh(\pi t)}=X^{it}\exp(-t/T)+O(\exp(-\pi t)),
\end{equation}
\begin{equation}\label{phi0 expression}
\varphi_0=\frac{-\cosh\beta}{2\pi^2\sinh^2\beta},
\end{equation}
\begin{equation}\label{phiBbound}
\varphi_B(x)\ll X^{-1/2}\min(x,x^{1/2}),
\end{equation}
\begin{equation}\label{phiBseries}
\varphi_B(x)=\frac{2}{\pi}\sum_{k=1}^{\infty}(-1)^k (2k-1)\exp(-(2k-1)\beta)J_{2k-1}(x).
\end{equation}
\end{lem}
\begin{proof}
It follows from \cite[Eq. 6.6621.1]{GR}, \cite[Eq. 15.4.18]{HMF} that
\begin{equation}\label{integral J2it exp}
\int_0^{\infty}J_{2it}(x)\exp(ix\cosh\beta)dx=-\frac{\exp(-(\pi+2i\beta)t)}{i\sinh\beta}.
\footnote{Note that there is a typo in \cite[Eq. 7.7]{DesIw}: the minus sign is missed. For this reason, the formula \eqref{integral J0 exp} has an additional minus sign compared to \cite[Eq. 7.8]{DesIw}. Consequently, we obtain the same expression \eqref{phi0 expression} for $\varphi_0$  as in \cite[Eq. 7.8]{DesIw} regardless of the sign change in the definition of $\varphi(x)$. }
\end{equation}
Using \eqref{phiHat-Def}, \eqref{phi def} and \eqref{integral J2it exp} we prove \eqref{phihat expression}. Differentiating equation \eqref{integral J2it exp} with respect to $\beta$ and taking $t=0$ we obtain
\begin{equation}\label{integral J0 exp}
\int_0^{\infty}J_{0}(x)\exp(ix\cosh\beta)xdx=-\frac{\cosh\beta}{\sinh\beta}.
\end{equation}
Now \eqref{phi0 expression} follows from \eqref{phi0-Def},\eqref{phi def} and \eqref{integral J0 exp}.
Estimate \eqref{phiBbound} can be proved in the same way as \cite[Lemma 11]{DesIw}. Namely,
\begin{equation*}
\varphi_B(x)\ll |\sinh(2\beta)|\int_0^1\frac{\min(\xi x,\sqrt{\xi x})}{|\cosh^2\beta-\xi^2|^{3/2}}d\xi\ll
X^{-1/2}\min(x,x^{1/2}).
\end{equation*}
To prove \eqref{phiBseries} it is required to substitute \eqref{phi def} to \eqref{phiB series Def} and use \cite[Eq. 6.6621.1]{GR}, \cite[Eq. 15.4.18]{HMF}.
\end{proof}

\section{Initial steps of the proof}\label{sec:initial steps}
To prove Theorem \ref{thm:spec.exp.sum new bound} we follow the approach of Iwaniec \cite{IwPG}. See also \cite[Section 6]{LuoSarnakPG} for more details.
First of all, the problem can be reduced to the analysis of the sum
\begin{equation}\label{spec.exp.sum.smooth}
\sum_{j}X^{it_j}\exp(-t_j/T)
\end{equation}
at the cost of the error $O(T\log T).$ To this end, we introduce a smooth function $g(x)$   such that $g(x)=1$ for $1\le x\le T$ and
$g(x)=0$ for $x\le1/2$ and $x\ge T+1/2.$ Then
\begin{equation*}
\sum_{t_j\le T}X^{it_j}=\sum_{j}X^{it_j}\exp(-t_j/T)g(t_j)\exp(t_j/T)+O(T).
\end{equation*}
Let
\begin{equation*}
\hat{g}(x)=\int_{-\infty}^{\infty}g(\xi)\exp(\xi/T)e(x\xi)d\xi.
\end{equation*}
According to \cite[Section 6]{LuoSarnakPG}
\begin{equation}\label{g(x) estimate}
\hat{g}(x)\ll\min(T,1/|x|).
\end{equation}

Finally we have the following lemma (see \cite[Section 6]{LuoSarnakPG}).

\begin{lem} One has
\begin{equation}\label{spec.sum to spec.sum.exp}
\sum_{t_j\le T}X^{it_j}=
\int_{-1}^1\hat{g}(\xi)\left(\sum_{j}(X\exp(-2\pi\xi))^{it_j}\exp(-t_j/T)\right)d\xi+O(T\log T).
\end{equation}
\end{lem}

Let $h(x)$ be a smooth function supported in $[N,2N]$ such that
\begin{equation}\label{h conditions}
|h^{(j)}(x)|\ll N^{-j},\,\hbox{for}\,j=0,1,2,\ldots\, \int_{-\infty}^{\infty}h(x)dx=N.
\end{equation}

The first idea of Iwaniec was to investigate the following expression
\begin{equation}\label{spec.sum averaged over n}
\sum_n h(n)\sum_{j}\frac{|\rho_j(n)|^2}{\cosh(\pi t_j)}X^{it_j}\exp(-t_j/T).
\end{equation}
The sum over $n$ can be evaluated using the  Mellin inversion formula for $h(n)$ and the fact that Rankin zeta function \eqref{Rankin zeta} has a pole at the point $s=1$ with the residue given by equation \eqref{residue of Rankin zeta}. As in \cite[Lemma 8]{IwPG} one has
\begin{equation}\label{sum over n with h}
\sum_n h(n)\frac{|\rho_j(n)|^2}{\cosh(\pi t_j)}=
\frac{2}{\zeta(2)}N+\frac{1}{2\pi i}\int_{(1/2)}\tilde{h}(s)\frac{L(u_{j}\otimes u_{j},s)}{\cosh(\pi t_j)}ds.
\end{equation}
Substituting \eqref{sum over n with h} to \eqref{spec.sum averaged over n}, we obtain
\begin{multline}\label{spec.sum decomposition1}
\sum_{j}X^{it_j}\exp(-t_j/T)=
\frac{\zeta(2)}{2N}\sum_n h(n)\sum_{j}\frac{|\rho_j(n)|^2}{\cosh(\pi t_j)}X^{it_j}\exp(-t_j/T)\\-
\frac{\zeta(2)}{2N}
\frac{1}{2\pi i}\int_{(1/2)}\tilde{h}(s)\sum_{j}X^{it_j}\exp(-t_j/T)\frac{L(u_{j}\otimes u_{j},s)}{\cosh(\pi t_j)}ds.
\end{multline}
The standard tool for studying the first sum on the right-hand side of equation  \eqref{spec.sum decomposition1} is the Kuznetsov trace formula.

The second idea of Iwaniec was not to apply Kuznetsov trace formula directly but to find a function $\varphi(x)$ such that its transform \eqref{phiHat-Def} approximates $X^{it}\exp(-t/T)$ with a small error term. The procedure of finding a suitable $\varphi(x)$ is not straightforward and several examples of such functions are available in the literature, i.e. \cite[Eq. ~29]{IwPG}, \cite[Eq.~7.2]{DesIw}. The approximation by \cite[Eq.~7.2]{DesIw} produces a smaller error term, and therefore,  similarly to \cite{LuoSarnakPG} we choose $\varphi(x)$ defined by \eqref{phi def}. One can compute $\hat{\varphi}(t)$ explicitly, see equation \eqref{phihat expression}. Replacing  $X^{it}\exp(-t/T)$ by $\hat{\varphi}(t)$  in \eqref{spec.sum decomposition1}, we have
\begin{multline}\label{spec.sum decomposition2}
\sum_{j}X^{it_j}\exp(-t_j/T)+O(1)=
\frac{\zeta(2)}{2N}\sum_n h(n)\sum_{j}\frac{|\rho_j(n)|^2}{\cosh(\pi t_j)}\hat{\varphi}(t_j)\\-
\frac{\zeta(2)}{2N}
\frac{1}{2\pi i}\int_{(1/2)}\tilde{h}(s)\sum_{j}\hat{\varphi}(t_j)\frac{L(u_{j}\otimes u_{j},s)}{\cosh(\pi t_j)}ds.
\end{multline}
Applying Kuznetsov's trace formula \eqref{eq:KuzTrForm} to the first sum over $j$ on the right-hand side of equation \eqref{spec.sum decomposition2} and arguing as in the paper of Luo-Sarnak  \cite[p. 234]{LuoSarnakPG}, we obtain
\begin{multline}\label{spec.sum decomposition3}
\sum_{j}X^{it_j}\exp(-t_j/T)\ll
\frac{NX^{1/2}\log T}{(\max(1,NX^{1/2}/T))^{1/2}}+
\sqrt{N/X}\log^2N\\+T\log^2T+
\frac{1}{N}
\Biggl|\int_{(1/2)}\tilde{h}(s)\sum_{j}\hat{\varphi}(t_j)\frac{L(u_{j}\otimes u_{j},s)}{\cosh(\pi t_j)}ds\Biggr|.
\end{multline}

\begin{rem}
Estimate \eqref{spec.sum decomposition3} coincides with the  bound of Luo and Sarnak if $\max(1,NX^{1/2}/T)=NX^{1/2}/T.$  See \cite[p. 235, lines 2-4]{LuoSarnakPG}. This maximum arises naturally when we estimate the sum of Kloosterman sums
\begin{equation*}
S_n(\phi):=\sum_{q=1}^{\infty}\frac{S(n,n;q)}{q}\varphi\left( \frac{4\pi n}{q}\right)\ll
\sum_{q=1}^{\infty}\frac{|S(n,n;q)|}{q}X^{1/2}\frac{n}{q}\exp(-4\pi an/q)
\end{equation*}
using the Weil bound. Since $a\asymp X^{1/2}/(4T)$ (see \eqref{a,b def}), the sum over $q$  can be viewed as a sum over $q>\max(1,NX^{1/2}/T).$
\end{rem}

To analyse the integral in \eqref{spec.sum decomposition3}, Luo and Sarnak used the following "mean Lindel\"{o}f" estimate \cite[Eq. 5]{LuoSarnakPG}
\begin{equation}\label{mean Lindelof}
\sum_{t_j\le T}\frac{L(u_{j}\otimes u_{j},1/2+ir)}{\cosh(\pi t_j)}\ll (1+|r|)^4T^{2}\log^2T.
\end{equation}
In order to improve their result we consider the whole sum
\begin{equation*}
\sum_{j}\hat{\varphi}(t_j)\frac{L(u_{j}\otimes u_{j},1/2+ir)}{\cosh(\pi t_j)},
\end{equation*}
trying to use the oscillations of the function $\hat{\varphi}(t_j).$

It follows from the properties of $h(x)$ that $|\tilde{h}(s)|\ll N^{1/2}(1+|r|)^{-A}$ for any $A$ and $s=1/2+ir.$ Thus applying equation \eqref{mean Lindelof} to handle $|r|>T^{\epsilon}:=r_0$, we need to work only with $|r|\le r_0.$
Finally, using equations \eqref{Rankin to symsquare} and \eqref{alphaj}, we obtain
\begin{multline}\label{spec.sum decomposition10}
\sum_{j}X^{it_j}\exp(-t_j/T)\ll
\frac{NX^{1/2}\log T}{(\max(1,NX^{1/2}/T))^{1/2}}+
\sqrt{N/X}\log^2N+\\ T\log^2T+
\frac{1}{N^{1/2}}
\int_{-r_0}^{r_0}\Biggl|\frac{\zeta(1/2+ir)}{\zeta(1+2ir)}
\sum_{j}\alpha_j\hat{\varphi}(t_j)L(\sym^2 u_{j},1/2+ir)\Biggr|dr.
\end{multline}

\section{Exact formula}\label{sec:exactformula}
This section is devoted to the analysis of the first moment
\begin{equation}\label{firstmomentdef}
M_1(s):=\sum_{j}\alpha_j\hat{\varphi}(t_j)L(\sym^2 u_{j},s)
\end{equation}
with $s=1/2+ir,|r|\le T^{\epsilon}.$ The key ideas are similar to the ones used in \cite{BF1} for the holomorphic case. The main difference is that instead of the Petersson trace formula we now work with the Kuznetsov trace formula \eqref{eq:KuzTrForm}.

In order to prove an exact formula for $M_1(1/2+ir)$, we apply the technique of analytic continuation. For $\Re{s}>3/2$ define
\begin{equation}\label{sigmadef}
\Sigma(s)=\zeta(2s)\sum_{n=1}^{\infty}\frac{1}{n^{s}}\sum_{q=1}^{\infty}\frac{S(1,n^2;q)}{q}\varphi\left(\frac{4\pi n}{q}\right),
\end{equation}
\begin{equation}\label{sigmaBdef}
\Sigma_B(s)=\zeta(2s)\sum_{n=1}^{\infty}\frac{1}{n^{s}}\sum_{q=1}^{\infty}\frac{S(1,n^2;q)}{q}\varphi_B\left(\frac{4\pi n}{q}\right).
\end{equation}
Convergence of the double series in \eqref{sigmadef} and \eqref{sigmaBdef} for $\Re{s}>3/2$ follows from \eqref{Weilbound}, \eqref{phi def} and \eqref{phiBbound}.
\begin{lem} For $\Re{s}>3/2$ one has
\begin{equation}\label{firstmoment=MT+Sigma-SigmaB-ContSpectr}
M_1(s)=\zeta(2s)\varphi_0+\Sigma(s)-\Sigma_B(s)-
\frac{2\zeta(s)}{\pi}\int_{0}^{\infty}
\frac{\zeta(s+2it)\zeta(s-2it)}{|\zeta(1+2it)|^2}
\hat{\varphi}(t)dt.
\end{equation}
\end{lem}
\begin{proof}
Applying Kuznetsov trace formula \eqref{eq:KuzTrForm} and using identity \eqref{series with tau}, we obtain
\begin{multline*}\label{firstmoment1}
M_1(s)=\zeta(2s)\sum_{n=1}^{\infty}\frac{1}{n^{s}}\sum_j\alpha_j\hat{\varphi}(t_j)\lambda_j(1)\lambda_j(n^2)=\\
\zeta(2s)\varphi_0-
\frac{2\zeta(s)}{\pi}\int_{0}^{\infty}
\frac{\zeta(s+2it)\zeta(s-2it)}{|\zeta(1+2it)|^2}
\hat{\varphi}(t)dt+\\
\zeta(2s)\sum_{n=1}^{\infty}\frac{1}{n^{s}}\sum_{q=1}^{\infty}\frac{S(1,n^2;q)}{q}\varphi_H\left( \frac{4\pi n}{q}\right).
\end{multline*}
 Then equation \eqref{phi=phiB+phiH} yields the lemma.
\end{proof}
In order to extend exact formula \eqref{firstmoment=MT+Sigma-SigmaB-ContSpectr} to the critical line $\Re{s}=1/2$, it is required to continue analytically
double sums \eqref{sigmadef} and \eqref{sigmaBdef}. This is the subject of the two subsequent sections.
\section{Analysis of $\Sigma(s)$}\label{sec:sigma1}
The strategy of working with $\Sigma(s)$ is the same as in the proof of \cite[Lemma 5.1]{BF1}. First, we change the order of summation using the fact that $\Re{s}>3/2$.  Second, we apply the Mellin inverse formula for $\varphi(x)$, open the Kloosterman sum and obtain the Lerch zeta function.  After that we move the line of integration to the region, where the functional equation for the Lerch zeta function can be applied.

However, it turns out that the Mellin transform of  $\varphi(x)$ does not allow moving the line of integration to the desired region. The reason is that the function $\varphi(x)$  behaves asymptotically like $x$ when $x\rightarrow0$, and this is insufficient for absolute convergence. To overcome this difficulty we use the so-called "Hecke trick".  Accordingly, instead of working directly with  $\varphi(x)$ and $\Sigma(s)$, we introduce for a complex variable $\lambda$ with $\Re{\lambda}>1$ two functions
\begin{equation}\label{philambda def}
\varphi(\lambda,x)=\frac{\sinh\beta}{\pi}x^{\lambda}\exp(ix\cosh\beta),
\end{equation}
\begin{equation}\label{sigmalambda def}
\Sigma(\lambda,s)=\sum_{n=1}^{\infty}\frac{1}{n^{s}}\sum_{q=1}^{\infty}\frac{S(1,n^2;q)}{q}\varphi\left(\lambda,\frac{4\pi n}{q}\right),
\end{equation}
where $\beta$ is defined by \eqref{beta def}. We first prove a formula for $\Sigma(\lambda,s)$ under the assumption that $\Re{\lambda}>\Re{s}$  and then extend it by analytic continuation to the point $\lambda=1.$

\begin{lem} For $\Re{\lambda}>\Re{s}>3/2$ one has
\begin{equation}\label{Sigma(lambda,s)}
\Sigma(\lambda,s)=(4\pi)^{s-1}\tilde{\varphi}(\lambda,1-s)\mathscr{L}_{-4}(s)+
2(2\pi)^{s-1}\sum_{n=1}^{\infty}\frac{1}{n^{1-s}}\mathscr{L}_{n^2-4}(s)I\left(\lambda,\frac{n}{2}\right),
\end{equation}
where
\begin{equation}\label{eq:integralI}
I(\lambda,x):=\frac{1}{2\pi i}\int_{(\Delta)}\tilde{\varphi}(\lambda,w)\Gamma(1-s-w)\sin\left( \pi \frac{s+w}{2}\right)x^wdw
\end{equation}
with $-\Re{\lambda}<\Delta<-\Re{s}.$
\end{lem}
\begin{proof}
To change the order of summation in \eqref{sigmalambda def}, absolute convergence of the both series is required. Applying  \eqref{beta def}, \eqref{c def} and \eqref{a,b def}, we have
\begin{equation}\label{philambda bound}
|\varphi(\lambda,x)|\ll X^{1/2}x^{\Re(\lambda)}\exp(-ax).
\end{equation}
Using \eqref{Weilbound} we obtain that both sums in \eqref{sigmalambda def} are absolutely convergent for $\Re(s)>3/2.$
According to \cite[p. 312, Eq. 1]{BE}, the Mellin transform of $\varphi(\lambda,x)$  for $\Re{w}>-\Re{\lambda}$ is equal to
\begin{equation}\label{philambdaMellin}
\tilde{\varphi}(\lambda,w)=\frac{\sinh\beta}{\pi}\frac{\Gamma(w+\lambda)}{c^{w+\lambda}}.
\end{equation}
 It follows that
\begin{equation*}
\Sigma(\lambda,s)=\zeta(2s)\sum_{q=1}^{\infty}\frac{1}{q}
\frac{1}{2\pi i}
\int_{(\Delta)}\tilde{\varphi}(\lambda,w)\sum_{n=1}^{\infty}\frac{S(1,n^2;q)}{n^{w+s}}
\left(\frac{q}{4\pi} \right)^wdw.
\end{equation*}
We assume that $\max(-\Re{\lambda},1-\Re{s})<\Delta<-1/2$ to guarantee absolute convergence of the integral over $w$ and the sums over $q,n.$ Note that due to \eqref{argc} the function $\tilde{\varphi}(\lambda,w)$ is of exponential decay in terms of $\Im{w}.$
By \cite[Lemma 5.1]{BF1} one has
\begin{equation*}
\sum_{n=1}^{\infty}\frac{S(1,n^2;q)}{n^{w+s}}=
\sum_{d\Mod{q}}\frac{S(1,d^2;q)}{q^{w+s}}\zeta\left(\frac{d}{q},0,w+s\right).
\end{equation*}
Note that for all $d$, the Lerch zeta function has a simple pole at $w=1-s$ with residue one. We move the $w$-contour to the left up to $\Delta_1:=-s-\epsilon$, crossing a simple pole at $w=1-s.$ For the resulting integral we apply functional equation \eqref{LerchFE}. Finally,
\begin{multline*}
\Sigma(\lambda,s)=\zeta(2s)(4\pi)^{s-1}\tilde{\varphi}(\lambda,1-s)\sum_{q=1}^{\infty}\frac{1}{q^{1+s}}
\sum_{d\Mod{q}}S(1,d^2;q)+\\
\zeta(2s)\sum_{q=1}^{\infty}\frac{1}{q}
\frac{1}{2\pi i}
\int_{(\Delta_1)}\tilde{\varphi}(\lambda,w)
\frac{2(2\pi)^{s+w-1}}{q^{s+w}}\Gamma(1-s-w)\times \\\sin\left(\pi \frac{s+w}{2}\right)
\sum_{d\Mod{q}}S(1,d^2;q)\zeta(0,d/q;1-s-w)
\left(\frac{q}{4\pi} \right)^wdw.
\end{multline*}
Opening the Lerch zeta function and using \cite[Lemma 4.1]{BF1} we prove the lemma.
\end{proof}

\begin{lem} The following relation holds
\begin{multline}\label{integralI(lambda,x)rational}
I(\lambda,x)=\frac{\sinh\beta}{\pi c^{\lambda}}2^{\lambda-s}
\frac{\Gamma((1+\lambda-s)/2)\Gamma(1+(\lambda-s)/2)}{2\Gamma(1/2)}\\ \times \left(\frac{x}{c}\right)^{1-s}
\left(
\left(1+\sqrt{-x^2/c^2}\right)^{s-\lambda-1}+
\left(1-\sqrt{-x^2/c^2}\right)^{s-\lambda-1}
\right).
\end{multline}
\end{lem}
\begin{proof}
Moving the contour of integration in \eqref{eq:integralI} to the right, we cross simple poles at $w=1-s+j$, $j=0,1,2,\ldots$
Therefore,
\begin{equation*}
I(\lambda,x)=\frac{\sinh\beta}{\pi c^{\lambda}}\left(\frac{x}{c}\right)^{1-s}
\sum_{j=0}^{\infty}\frac{(-1)^j}{(2j)!}\Gamma(1+\lambda-s+2j)\left(\frac{x}{c}\right)^{2j}.
\end{equation*}
Using the duplication formula twice
\begin{equation*}
\Gamma(1+\lambda-s+2j)=\frac{2^{\lambda-s+2j}}{\sqrt{\pi}}\Gamma((1+\lambda-s)/2+j)\Gamma(1+(\lambda-s)/2)+j),
\end{equation*}
\begin{equation*}
\Gamma(2j+1)=\frac{2^{2j}}{\sqrt{\pi}}\Gamma(j+1/2)\Gamma(j+1),
\end{equation*}
we obtain
\begin{multline}\label{integralI(lambda,x)hypergeometric}
I(\lambda,x)=\frac{\sinh\beta}{\pi c^{\lambda}}2^{\lambda-s}
\frac{\Gamma((1+\lambda-s)/2)\Gamma(1+(\lambda-s)/2)}{\Gamma(1/2)}\\ \times \left(\frac{x}{c}\right)^{1-s} {}_2F_{1}\left(\frac{1+\lambda-s}{2},1+\frac{\lambda-s}{2},\frac{1}{2};-\frac{x^2}{c^2} \right).
\end{multline}
Finally,  \cite[Eq. 15.4.7]{HMF} yields the lemma.
\end{proof}

Analyzing the right-hand side of equation \eqref{Sigma(lambda,s)} we see that $\Sigma(\lambda,s)$ can be continued analytically to the region of convergence of the series
\begin{equation*}
\sum_{n=1}^{\infty}\frac{1}{n^{1-s}}\mathscr{L}_{n^2-4}(s)I\left(\lambda,\frac{n}{2}\right).
\end{equation*}
In view of  \eqref{integralI(lambda,x)rational}  and \eqref{eq:subconvexity}, we prove analytic continuation of $\Sigma(\lambda,s)$ to the region of our interest, namely $\Re{s}=1/2$ and $\lambda=1.$


\begin{lem} For $\Re{s}=1/2$ one has
\begin{equation}\label{Sigma(1,s)}
\Sigma(s)=(4\pi)^{s-1}\tilde{\varphi}(1,1-s)\mathscr{L}_{-4}(s)+
2(2\pi)^{s-1}\sum_{n=1}^{\infty}\frac{1}{n^{1-s}}\mathscr{L}_{n^2-4}(s)I\left(\frac{n}{2}\right),
\end{equation}
where
\begin{multline}\label{integralI(1,x)}
I(x)=\frac{\sinh\beta}{\pi c}2^{1-s}\left(\frac{x}{c}\right)^{1-s}
\frac{\Gamma(1-s/2)\Gamma(3/2-s/2)}{\Gamma(1/2)}\times \\
\left(
\left(1+\sqrt{-x^2/c^2}\right)^{s-2}+
\left(1-\sqrt{-x^2/c^2}\right)^{s-2}
\right).
\end{multline}
\end{lem}


\section{Analysis of $\Sigma_B(s)$}\label{sec:sigma2}

Substituting expression \eqref{phiBseries} to formula \eqref{sigmaBdef} we obtain
\begin{equation}\label{sigmaBwith sum over k}
\Sigma_B(s)=\frac{1}{\pi^2}\sum_{k=1}^{\infty} (2k-1)\exp(-(2k-1)\beta)\Sigma_B(k,s),
\end{equation}
where
\begin{equation}\label{sigmaB(k,s) def}
\Sigma_B(k,s)=2\pi(-1)^k
\zeta(2s)\sum_{n=1}^{\infty}\frac{1}{n^{s}}\sum_{q=1}^{\infty}\frac{S(1,n^2;q)}{q}J_{2k-1}\left(\frac{4\pi n}{q}\right).
\end{equation}
The analysis of the function $\Sigma_B(k,s)$ was performed in \cite[Lemma 5.1]{BF1}. The main difference is that  in \cite[Lemma 5.1]{BF1} we assumed that $k\ge6$ and now the sum is over all $k\geq 1$. The case $k=1$ is the most difficult one. In all other cases we can proceed exactly as in the proof of \cite[Lemma 5.1]{BF1}. The major problem for $k=1$ is that the line of integration cannot be moved to $\Delta_1:=-s-\epsilon$  because it is required that $1-2k<-s-\epsilon<-3/2-\epsilon$, which is possible only for $k\geq 2.$ To avoid this problem we again use the "Hecke trick" in the same manner as in \cite{BykF}. More precisely, we replace $k$ in the order of the Bessel function by a parameter $\lambda$ and assume that $\Re{\lambda}>5/4.$  In doing so we can carry on the analysis of \cite[Lemma 5.1]{BF1} and obtain for $\Re{\lambda}>(1+\Re{s})/2, \Re{s}>3/2$ the following expression
\begin{equation}\label{sigmaB(lambda,s)continuation}
\Sigma_B(\lambda,s)=(2\pi)^{s}i^{2k}\frac{\Gamma(\lambda-s/2)}{\Gamma(\lambda+s/2)}\mathscr{L}_{-4}(s)+
(2\pi)^{s}i^{2k}\sum_{n=1}^{\infty}\frac{1}{n^{1-s}}\mathscr{L}_{n^2-4l^2}(s)I_B\left(\lambda,n\right),
\end{equation}
where
\begin{equation}\label{eq:integralIB}
I_B(\lambda,x):=\frac{1}{2\pi i}\int_{(\Delta)}\frac{\Gamma(\lambda-1/2+w/2)}{\Gamma(\lambda+1/2-w/2)}\Gamma(1-s-w)\sin\left( \pi \frac{s+w}{2}\right)x^wdw
\end{equation}
with $1-2\Re{\lambda}<\Delta<1-\Re{s}.$ As in \cite[Lemma 5.1]{BF1} we prove that for $x \geq 2$ the following holds
\begin{multline}\label{integralIBgeq2}
I_B(\lambda,x)=
\frac{2^{2\lambda}}{2^s\sqrt{\pi}}\cos\left( \pi\lambda-\frac{\pi s}{2}\right)
\frac{\Gamma(\lambda-s/2)\Gamma(\lambda+1/2-s/2)}{\Gamma(2\lambda)} x^{1-2\lambda} \\ \times {}_2F_{1}\left(\lambda-s/2,\lambda+1/2-s/2,2\lambda;\frac{4}{x^2} \right).
\end{multline}
The right-hand side of \eqref{sigmaB(lambda,s)continuation} yields analytic continuation of $\Sigma_B(\lambda,s)$ to the region of convergence of the series
\begin{equation*}
\sum_{n=1}^{\infty}\frac{1}{n^{1-s}}\mathscr{L}_{n^2-4l^2}(s)I_B\left(\lambda,n\right).
\end{equation*}
Using \eqref{integralIBgeq2}, \eqref{eq:subconvexity} we obtain analytic continuation of $\Sigma_B(\lambda,s)$ to $\Re{s}=1/2, \lambda=1.$


\begin{lem} For $\Re{s}=1/2$ the following formula holds
\begin{multline}\label{SigmaB(1,s)}
\Sigma_B(s)=\frac{1}{\pi^2}\sum_{k=1}^{\infty} (2k-1)\exp(-(2k-1)\beta)\Biggl(
\frac{(2\pi)^{s}i^{2k}}{2}\frac{\Gamma(k-s/2)}{\Gamma(k+s/2)}\mathscr{L}_{-4}(s)\\+
\frac{(2\pi)^s}{\sqrt{\pi}}\zeta(2s-1)\cos\left(\frac{\pi s}{2}\right)
\frac{\Gamma(k-s/2)\Gamma(k+1/2-s/2)}{\Gamma(k+s/2)\Gamma(k-1/2+s/2)}\Gamma(s-1/2)\\+
\frac{(2\pi)^{s}\sin(\pi s/2)}{\sqrt{\pi}}\mathscr{L}_{-3}(s)\Phi_k(s,1/4)\\+
\frac{\pi^{s}\cos(\pi s/2)}{\sqrt{\pi}}\sum_{n>2}\mathscr{L}_{n^2-4}(s)
 n^s\Psi_k(s,4/n^2)
\Biggr),
\end{multline}
where
\begin{equation}\label{defPhik}
\Phi_k(s,x):=
\frac{\Gamma(k-s/2)\Gamma(1-k-s/2)}{\Gamma(1/2)}
{}_2F_1\left( k-\frac{s}{2},1-k-\frac{s}{2},\frac{1}{2};\frac{1}{4}\right),
\end{equation}
\begin{equation}\label{defPsik}
\Psi_k(s,x):=x^k
\frac{\Gamma(k-s/2)\Gamma(k+1/2-s/2)}{\Gamma(2k)}
{}_2F_1\left( k-\frac{s}{2},k+\frac{1-s}{2},2k;x\right).
\end{equation}
\end{lem}

\section{Special functions}\label{sec:specialfunctions}
In order to obtain an upper bound for $M_1(1/2+ir)$, it is required  to estimate $I(x)$, $\Phi_k(s,x)$, $\Psi_k(s,x)$ defined by \eqref{integralI(1,x)}, \eqref{defPhik}, \eqref{defPsik} for $\Re{s}=1/2.$


\begin{lem}\label{lemma on I(x) estimate}
For $s=1/2+ir$ one has
\begin{equation}\label{I(x)estimate1}
I(x)\ll(1+|r|)^{3/2}
\left(\frac{x}{|c|}\right)^{1/2}
\left|1-\sqrt{-x^2/c^2}\right|^{-3/2}.
\end{equation}
\end{lem}
\begin{proof}
Let
\begin{equation}\label{gamma+-def}
\gamma_{\pm}=\arg{(1\pm\sqrt{-x^2/c^2})}.
\end{equation}
To estimate $I(x)$ given by \eqref{integralI(1,x)} we  use the relation
\begin{equation}\label{relation for |z^a|}
|z^a|=|z|^{\Re{a}}\exp(-\Im{a}\arg{z}),
\end{equation}
obtaining
\begin{multline*}\label{I(x)estimate2}
I(x)\ll(1+|r|)^{3/2}\left(\frac{x}{|c|}\right)^{1/2}\exp(-\pi|r|/2)\\ \times
\left(
\frac{\exp(-r(\arg{c}+\gamma_{-}))}
{\left|1-\sqrt{-x^2/c^2}\right|^{3/2}}+
\frac{\exp(-r(\arg{c}+\gamma_{+}))}
{\left|1+\sqrt{-x^2/c^2}\right|^{3/2}}
\right).
\end{multline*}
Thus to prove \eqref{I(x)estimate1} we need to show  that
\begin{equation}\label{relation on gamma+gamma+-}
-\pi|r|/2-r(-\pi/2+\gamma+\gamma_{\pm})\le0.
\end{equation}
It follows from \eqref{argc} that
\begin{equation}\label{equation for smth}
1\pm\sqrt{-x^2/c^2}=1\pm\frac{x}{|c|}\cos\gamma\mp i\frac{x}{|c|}\sin\gamma.
\end{equation}
In the plus case we have $-\gamma<\gamma_+<0$ and thus \eqref{relation on gamma+gamma+-} is satisfied.
In the minus case we have $0<\gamma_{-}<\pi-\gamma$ and  \eqref{relation on gamma+gamma+-} is also satisfied.
\end{proof}


\begin{lem} \label{lemma Sigma estimate}
For $s=1/2+ir,$  one has
\begin{equation}\label{Sigma(s) estmate}
\Sigma(s)\ll
\frac{(1+|r|)^AX^{\theta}}{X^{1/4}}\left(
X^{1/2}T^{1/2}+T^{3/2}\right),
\end{equation}
where $A$ is some positive constant.
\end{lem}
\begin{proof}
Estimating \eqref{Sigma(1,s)} by absolute value and using \eqref{eq:subconvexity} and \eqref{I(x)estimate1} we obtain
\begin{equation}\label{Sigma(s) estmate2}
\Sigma(s)\ll\frac{(1+|r|)^A}{|c|^{1/2}}\sum_{n=1}^{\infty}n^{2\theta}
\left|1-\sqrt{-n^2/(4c^2)}\right|^{-3/2}.
\end{equation}
In view of \eqref{equation for smth} one has
\begin{equation*}
\left|1-\sqrt{-x^2/c^2}\right|^2=
\left(1-\frac{x}{|c|}\right)^2+2\frac{x}{|c|}\sin^2(\gamma/2).
\end{equation*}
Consequently,
\begin{equation}\label{Sigma(s) estmate3}
\Sigma(s)\ll\frac{(1+|r|)^AX^{\theta}}{X^{1/4}}\sum_{n=1}^{\infty}f\left(\frac{n}{2|c|}\right),
\end{equation}
where
\begin{equation}\label{def of f(x) delta}
f(x)=
\frac{x^{2\theta}}{\left((1-x)^2+x\delta\right)^{3/4}}, \quad \delta=2\sin^2(\gamma/2)\asymp\frac{1}{8T^2}.
\end{equation}
Calculating the first derivative of $f(x)$ and solving the equation $f'(x)=0$, we show that the function $f(x)$ has only one critical point $x_0=1+O(\delta)$ belonging to $(0,\infty).$ Therefore,
\begin{multline}\label{Sigma(s) estmate4}
\sum_{n=1}^{\infty}f\left(\frac{n}{2|c|}\right)=\int_0^{\infty}f\left(\frac{x}{2|c|}\right)dx+O(\max_{0<x<\infty}|f(x)|)\ll\\
X^{1/2}\int_0^{1-\delta^{1/2}}\frac{x^{2\theta}dx}{(1-x)^{3/2}}+
X^{1/2}\int_{1-\delta^{1/2}}^{1+\delta^{1/2}}\frac{x^{2\theta}dx}{(x\delta)^{3/4}}\\+
X^{1/2}\int_{1+\delta^{1/2}}^{\infty}\frac{x^{2\theta}dx}{(1-x)^{3/2}}+
T^{3/2}\ll
X^{1/2}T^{1/2}+T^{3/2}.
\end{multline}
Substituting  \eqref{Sigma(s) estmate4} to \eqref{Sigma(s) estmate3} yields the required result.
\end{proof}


It is not obvious how to improve  \eqref{Sigma(s) estmate} for general $X$ and $T$. However, we can prove a better estimate in the special case when $T$ is sufficiently large and $\left\|X^{1/2}+X^{-1/2}\right\|\neq0,$ where $\|x\|$ denotes  the distance between $x$ and the nearest integer.  Let
\begin{equation}\label{def fractional part of 2|c|}
\fr(X)=\left\|X^{1/2}+X^{-1/2}\right\|.
\end{equation}


\begin{lem} \label{lemma Sigma estimate2}
Assume that $\fr(X)\neq0$ and
\begin{equation}\label{lower bound on T }
T\gg
\frac{X^{1/4+\epsilon}}{\fr^{1/2}(X)}.
\end{equation}
For $s=1/2+ir$  one has
\begin{equation*}\label{Sigma(s) estmate for big T}
\Sigma(s)\ll
\frac{(1+|r|)^AX^{\theta}}{X^{1/4}}\left(
X^{1/2}\min\left(T,\frac{X^{1/2}}{\fr(X)}\right)^{1/2}+
\min\left(T,\frac{X^{1/2}}{\fr(X)}\right)^{3/2}\right),
\end{equation*}
where $A$ is some positive constant.
\end{lem}
\begin{proof}
We would like to improve the estimate on the sum over $n$ in \eqref{Sigma(s) estmate3}.
Calculating the first derivative of $f(x)$ and solving the equation $f'(x)=0$ we obtain that the function $f(x)$ has only one critical point $x_0=1+O(\delta)$ belonging to $(0,\infty).$ According to \eqref{Sigma(s) estmate4} the largest contribution to the estimate on $\Sigma(s)$ comes from the sum over
\begin{equation*}
|2|c|-n|\ll2|c|T^{-1}.
\end{equation*}
Note that it may happen that for a sufficiently large $T$ this region does not contain any integer number.

First of all, we analyse the value of $\|2|c|\|.$ It follows from \eqref{c def} that
\begin{multline}\label{fractional part of 2|c|}
2|c|=2\left(\cosh^2(\log\sqrt{X})-\sin^2((2T)^{-1})\right)^{1/2}\\=
X^{1/2}+X^{-1/2}+O(X^{-1/2}T^{-2}).
\end{multline}

If $X^{1/2}+X^{-1/2}$ is an integer then $\|2|c|\|=O(X^{-1/2}T^{-2}).$ Since the sum over $n$ in \eqref{Sigma(s) estmate3} contains summands
\begin{equation*}
f\left(\frac{[2|c|]}{2|c|}\right)+f\left(\frac{[2|c|]+1}{2|c|}\right)\ll
\frac{1}{\left((\|2|c|\|/(2|c|))^2+\delta\right)^{3/4}}\ll T^{3/2},
\end{equation*}
the estimate \eqref{Sigma(s) estmate} cannot be improved in this case.

If $X^{1/2}+X^{-1/2}$ is not an integer, then for
\begin{equation}\label{lower bound1 on T }
T\gg X^{-1/4+\epsilon}\|X^{1/2}+X^{-1/2}\|^{-1/2}
\end{equation}
it follows that $\|2|c|\|$ is very close to $\|X^{1/2}+X^{-1/2}\|.$ The sum over $n$ in \eqref{Sigma(s) estmate3} can be decomposed into two sums
\begin{equation}\label{Sigma(s) estmate5}
\Sigma(s)\ll\frac{(1+|r|)^AX^{\theta}}{X^{1/4}}\left(
\sum_{n=1}^{[2|c|]}f\left(\frac{n}{2|c|}\right)+
\sum_{n=[2|c|]+1}^{\infty}f\left(\frac{n}{2|c|}\right)
\right).
\end{equation}
For $n\le[2|c|]$ the function $f(n/(2|c|))$ is increasing and for $n\ge[2|c|]+1$ it is decreasing provided that
\begin{equation}\label{lower bound2 on T}
T\gg X^{1/4+\epsilon}\|X^{1/2}+X^{-1/2}\|^{-1/2}.
\end{equation}
As a result,
\begin{equation*}
\sum_{n=1}^{[2|c|]}f\left(\frac{n}{2|c|}\right)=\int_{1}^{[2|c|]}f\left(\frac{x}{2|c|}\right)dx+
O\left(f\left(\frac{[2|c|]}{2|c|}\right)\right),
\end{equation*}
\begin{equation*}
\sum_{n=[2|c|]+1}^{\infty}f\left(\frac{n}{2|c|}\right)=\int_{[2|c|]+1}^{\infty}f\left(\frac{x}{2|c|}\right)dx+
O\left(f\left(\frac{[2|c|]+1}{2|c|}\right)\right).
\end{equation*}
Evaluating these integrals we obtain that the first sum can be bounded as follows
\begin{equation*}
\sum_{n=1}^{[2|c|]}f\left(\frac{n}{2|c|}\right)\ll
X^{1/2}\min\left(T^{1/2},\left(\frac{2|c|}{\|2|c|\|}\right)^{1/2}\right)+
\min\left(T^{3/2},\left(\frac{2|c|}{\|2|c|\|}\right)^{3/2}\right),
\end{equation*}
and the same estimate is also valid for the second sum.

\end{proof}


\begin{lem}\label{Psik lemma}
For $s=1/2+2ir$, $0<x<1$ the following estimates hold
\begin{equation}\label{Psik Res=1/2 estimate1}
\exp(\pi|r|)\Psi_k(s,x)\ll
(1-x)^{1/4}
\frac{(1+|r|)^{2k-2a-1}}{2^{2k}k^{a+1/2}}
\left(\frac{x}{1-x}\right)^{k-a},
\end{equation}
\begin{equation}\label{Psik Res=1/2 estimate2}
\exp(\pi|r|)\Psi_k(s,x)\ll
(1-x)^{1/4}
\frac{(1+|r|)^{2b}}{k^{2b-1/2}}
\left(\frac{x}{1-x}\right)^{b},
\end{equation}
where $a$, $b$ are some absolute constants such that $0<a<k-1/4,$ $1/4<b<k.$
\end{lem}
\begin{proof}
Using \cite[Eq. 15.8.1]{HMF} we obtain
\begin{multline*}
\Psi_k(s,x)=\frac{x^k}{(1-x)^{k-s/2}}
\frac{\Gamma(k-s/2)\Gamma(k+1/2-s/2)}{\Gamma(2k)}\\ \times
{}_2F_1\left( k-\frac{s}{2},k-\frac{1-s}{2},2k;-\frac{x}{1-x}\right).
\end{multline*}
Writing the Mellin-Barnes integral representation for the hypergeometric function \cite[Eq. 15.6.6]{HMF}, we have
\begin{multline}\label{Psik Mellin}
\Psi_k(s,x)=\frac{x^k}{(1-x)^{k-s/2}}
\frac{\Gamma(k+1/2-s/2)}{\Gamma(k-1/2+s/2)}\times\\
\frac{1}{2\pi i}\int_{\Delta}\frac{\Gamma(k-s/2+z)\Gamma(k-1/2+s/2+z)\Gamma(-z)}{\Gamma(2k+z)}\left(\frac{x}{1-x}\right)^zdz,
\end{multline}
where $1/4-k<\Delta<0$.
To prove \eqref{Psik Res=1/2 estimate1} we move the contour of integration to the line $\Re{z}=-a$ such that $0<a<k-1/4.$  Arguing in the same way as in \cite{Frol} we prove
\begin{equation}\label{Psik Mellin estimate1}
\Psi_k(s,x)\ll(1-x)^{1/4}(k+|r|)^{1/2}\left(\frac{x}{1-x}\right)^{k-a}
\frac{|\Gamma(k-1/4-a+ir)|^2}{\Gamma(2k-a)}.
\end{equation}
To simplify notation in estimates of the quotient of gamma factors, let  us assume that $a:=1/4+n$, where $n$ is a positive integer. Note that all our arguments are valid for an arbitrary $a.$  One has
\begin{multline*}
|\Gamma(k-1/2-n+ir)|=\prod_{j=1}^{k-n-2}|j+1/2+ir||\Gamma(3/2+ir)|\ll\\
(1+|r|)^{k-n-1}|\exp(-\pi|r|/2)\Gamma(k-1/2-n)
\end{multline*}
since
\begin{equation*}
|j+1/2+ir|\le j+1/2+|r|\le (j+1/2)(1+|r|),\,\hbox{for}\,j\ge1.
\end{equation*}
Therefore,
\begin{multline}\label{Gamma quotient estimate1}
\frac{|\Gamma(k-1/4-a+ir)|^2}{\Gamma(2k-a)}\ll
(1+|r|)^{2k-2a-3/2}|\exp(-\pi|r|)\frac{\Gamma^2(k-1/4-a)}{\Gamma(2k-a)}\\ \ll
\frac{(1+|r|)^{2k-2a-3/2}|\exp(-\pi|r|)}{2^{2k}k^{1+a}}.
\end{multline}
Substituting \eqref{Gamma quotient estimate1} to \eqref{Psik Mellin estimate1} we obtain \eqref{Psik Res=1/2 estimate1}.

To prove estimate \eqref{Psik Res=1/2 estimate2} we move the contour of integration in \eqref{Psik Mellin} to the line $\Re{z}=b-k$ such that $1/4<b<k.$ Consequently,
\begin{multline}\label{Psik Mellin estimate2}
\Psi_k(s,x)\ll
(1-x)^{1/4}
\left(\frac{x}{1-x}\right)^{b}
(k+|r|)^{1/2}\times\\
\int_{-\infty}^{\infty}\left|\frac{\Gamma(b-1/4+i(t-r))\Gamma(b-1/4+i(t+r))\Gamma(k-b-it)}{\Gamma(k+b+it)}\right|dt\ll\\
(1-x)^{1/4}
\left(\frac{x}{1-x}\right)^{b}
(k+|r|)^{1/2}\times\\
\int_{-\infty}^{\infty}\frac{\exp(-\pi(|t-r|+|t+r|)/2)(1+|t-r|)^{b-3/4}(1+|t+r|)^{b-3/4}}{(k+|t|)^{2b}}dt\ll\\
(1-x)^{1/4}
\left(\frac{x}{1-x}\right)^{b}
(k+|r|)^{1/2}\exp(-\pi|r|)\frac{(1+|r|)^{2b-1/2}}{k^{2b}}.
\end{multline}
This completes the proof of \eqref{Psik Res=1/2 estimate2}.
\end{proof}

\begin{lem}\label{Phik lemma}
For $s=1/2+2ir$ the following estimate holds
\begin{equation}\label{Phik Res=1/2 estimate}
\exp(\pi|r|)\Phi_k(s,1/4)\ll
\frac{4^k}{k+|r|}.
\end{equation}
\end{lem}
\begin{proof}
Using \cite[Eq. 15.8.1]{HMF} and the Mellin-Barnes integral for the hypergeometric function \cite[Eq. 15.6.6]{HMF}, we obtain
\begin{multline}\label{Phik Mellin}
\exp(\pi|r|)\Phi_k(s,1/4)=\exp(\pi|r|)
\frac{(4/3)^k\Gamma(1-k-s/2)}{\Gamma(k-1/2+s/2)}\times\\
\frac{1}{2\pi i}\int\frac{\Gamma(k-s/2+z)\Gamma(k-1/2+s/2+z)\Gamma(-z)}{\Gamma(1/2+z)}\frac{dz}{3^z}.
\end{multline}
Applying \cite[Eq. 5.5.3]{HMF} and moving the contour of integration in \eqref{Phik Mellin} to the line $\Re{z}=3/4-k$ yields the estimate
\begin{multline*}
\exp(\pi|r|)\Phi_k(s,1/4)\ll
\frac{4^k}{|\Gamma(k-1/4+ir)\Gamma(k+1/4+ir)|}\times\\
\int_{-\infty}^{\infty}\left|\frac{\Gamma(1/2+i(t-r))\Gamma(1/2+i(t+r))\Gamma(k-3/4-it)}{\Gamma(-k+5/4+it)}\right|dt.
\end{multline*}
Applying \cite[Eq. 5.5.3]{HMF} once again together with \cite[Eq. 5.4.4]{HMF} gives
\begin{equation}\label{Phik Mellin estimate1}
\exp(\pi|r|)\Phi_k(s,1/4)\ll
\frac{4^k}{|\Gamma(k+ir)|^2}
\int_{0}^{\infty}\frac{|\Gamma(k+it)|^2\exp(\pi t)}{(\cosh(2\pi t)+\cosh(2\pi r))^{1/2}(k+t)}dt.
\end{equation}
Using the methods of \cite[Section 2]{Frol} one can show that
\begin{equation}\label{estimate on |Gamma(k+it)|^2}
|\Gamma(k+it)|^2
\asymp\exp(-\pi|t|)
\begin{cases}
\Gamma^2(k)\exp(|t|(\pi-g_1(|t|/(k-1))),& \text{if } 0\le |t|\le k-1,\\
|t|^{2k-1}\exp(|t|g_2((k-1)/|t|)),& \text{if } k-1<|t|.\\
\end{cases}
\end{equation}
To estimate \eqref{Phik Mellin estimate1} we consider separately two cases: $k-1\le|r|$ and $k-1>|r|.$  For $k-1\le|r|$ using \eqref{estimate on |Gamma(k+it)|^2} we obtain
\begin{multline*}
\exp(\pi|r|)\Phi_k(s,1/4)\ll
\frac{4^k}{|\Gamma(k+ir)|^2}\\ \times \Biggl(
\Gamma^2(k)\exp(-\pi|r|)k^{-1}
\int_{0}^{k-1}\exp(t(\pi-g_1(t/(k-1)))dt+\\
\exp(-\pi|r|)\int_{k-1}^r\exp((2k-2)\log t+tg_2((k-1)/t))dt+\\
\int_r^{\infty}\exp(-\pi t+(2k-2)\log t+tg_2((k-1)/t))dt
\Biggr).
\end{multline*}
Estimating the integrals by the means of \cite[Lemma 3]{Frol} we show that
\begin{equation*}
\exp(\pi|r|)\Phi_k(s,1/4)\ll
\frac{4^k}{|\Gamma(k+ir)|^2}\exp(-\pi|r|)r^{2k-2}\exp(rg_2((k-1)/r)).
\end{equation*}
Applying \eqref{estimate on |Gamma(k+it)|^2} we finally prove that
\begin{equation*}
\exp(\pi|r|)\Phi_k(s,1/4)\ll 4^k/(1+|r|).
\end{equation*}
The case $k-1>|r|$ can be treated in the same way.
\end{proof}
\begin{rem}
In the case when $s=1/2$ one has (see \cite[Eq. 2.12]{BF1}) a much better bound  $\Phi_k(1/2,1/4)\ll k^{-1/2}.$ We expect that for $s=1/2+2ir$ the estimate of the following shape
\begin{equation}\label{Phik optimal estimate}
\exp(\pi|r|)\Phi_k(s,1/4)\ll(1+|r|)^a k^{-1/2}
\end{equation}
should be correct for some positive constant $a$.
\end{rem}

\begin{lem} \label{lemma SigmaB estimate}
For $s=1/2+2ir,$ $X>5$ one has
\begin{equation}\label{SigmaB(s) estmate}
\Sigma_B(s)\ll(1+|r|)^AX^{-1/2},
\end{equation}
where $A$ is some positive constant.
\end{lem}
\begin{proof}
To estimate \eqref{SigmaB(1,s)} we use \eqref{beta def}, \eqref{eq:subconvexity}, Lemmas \ref{Psik lemma} and \ref{Phik lemma}, together with some classical estimates on the Gamma function and the Riemann zeta function. Consequently,
\begin{multline}\label{sum over k of Phik}
\sum_{k=1}^{\infty} (2k-1)\exp(-(2k-1)\beta)
\frac{(2\pi)^{s}\sin(\pi s/2)}{\sqrt{\pi}}\mathscr{L}_{-3}(s)\Phi_k(s,1/4) \\ \ll
(1+|r|)^A\sum_{k=1}^{\infty}(2/\sqrt{X})^{2k-1}\ll (1+|r|)^AX^{-1/2}.
\end{multline}
To estimate the contribution of
\begin{equation*}
\sum_{k=1}^{\infty} (2k-1)\exp(-(2k-1)\beta)
\frac{\pi^{s}\cos(\pi s/2)}{\sqrt{\pi}}\sum_{n>2}\mathscr{L}_{n^2-4}(s)
 n^s\Psi_k(s,4/n^2)
\end{equation*}
we divide the sum over $n$ in two parts with the conditions $n=\le4(1+|r|)$ and $n>4(1+|r|)$. For the first sum we apply \eqref{Psik Res=1/2 estimate2} and obtain
\begin{equation*}
(1+|r|)^A\sum_{k=1}^{\infty}\frac{1}{k^{2b-3/2}X^{k-1/2}}\ll
(1+|r|)^AX^{-1/2}.
\end{equation*}
For the second sum over $n>4(1+|r|)$ we use the estimate \eqref{Psik Res=1/2 estimate1} and obtain
\begin{multline*}
(1+|r|)^{A_1}\sum_{k=1}^{\infty}\frac{k}{X^{k-1/2}}
\sum_{n>4(1+|r|)}\frac{n^{2\theta+\epsilon}}{(n^2/4-1)^{k-a-1/4}}
\frac{(1+|r|)^{2k-2a-1}}{2^{2k}k^{a+1/2}}\\ 
\ll (1+|r|)^AX^{-1/2}.
\end{multline*}
\end{proof}

\begin{rem}
We only need the condition $X>5$ in Lemma \ref{lemma SigmaB estimate}  to estimate the contribution of $\Phi_k(s,1/4).$ In particular, estimate \eqref{Phik optimal estimate} would result in a bound of the type
\begin{equation}\label{SigmaB(s) optimal estmate}
\Sigma_B(s)\ll\frac{(1+|r|)^A}{(X-1)^{B}}
\end{equation}
for any $X>1$ with some positive constants $A$ and $B$.
\end{rem}

\section{Estimates for the moment of Maa{\ss} $\sym^2$ $L$-functions}\label{sec:estimatesmoment}

\begin{lem} For $\Re{s}=1/2$ the following formula holds
\begin{multline}\label{firstmoment=MT+Sigma-SigmaB-ContSpectr Res=1/2}
M_1(s)=\zeta(2s)\varphi_0+\Sigma(s)-\Sigma_B(s)-\\
\frac{2\zeta(s)}{\pi}\int_{0}^{\infty}
\frac{\zeta(s+2it)\zeta(s-2it)}{|\zeta(1+2it)|^2}
\hat{\varphi}(t)dt-2\frac{\zeta(2s-1)}{\zeta(2-s)}\hat{\varphi}\left(\frac{1-s}{2i}\right),
\end{multline}
where $\Sigma(s)$ is given by \eqref{Sigma(1,s)} and $\Sigma_B(s)$ by \eqref{SigmaB(1,s)}.
\end{lem}
\begin{proof}
In order to prove that there exist an analytic continuation of \eqref{firstmoment=MT+Sigma-SigmaB-ContSpectr} to the critical line $\Re{s}=1/2$, it is only left to consider
\begin{multline}\label{continious spectr}
\frac{2\zeta(s)}{\pi}\int_{0}^{\infty}
\frac{\zeta(s+2it)\zeta(s-2it)}{|\zeta(1+2it)|^2}
\hat{\varphi}(t)dt\\ =
\frac{\zeta(s)}{2\pi i}\int_{(0)}
\frac{\zeta(s+z)\zeta(s-z)}{\zeta(1+z)\zeta(1-z)}
\hat{\varphi}\left(\frac{z}{2i}\right)dz.
\end{multline}
The continuation of $\Sigma(s)$ and $\Sigma_B(s)$ is given by \eqref{Sigma(1,s)} and \eqref{SigmaB(1,s)}, respectively.
Arguing in the same manner as in \cite[Theorem 7.3]{BF2} we obtain \eqref{firstmoment=MT+Sigma-SigmaB-ContSpectr Res=1/2}.
\end{proof}

 The procedure of analytic continuation of \eqref{firstmoment=MT+Sigma-SigmaB-ContSpectr Res=1/2} to the critical point $s=1/2$ is not straightforward since there are two summands on the right-hand side of \eqref{firstmoment=MT+Sigma-SigmaB-ContSpectr Res=1/2}
with simple poles at $s=1/2$.
\begin{lem} The following formula holds at the critical point
\begin{multline}\label{firstmoment=MT+Sigma-SigmaB-ContSpectr s=1/2}
M_1(1/2)=-\frac{1}{2\pi^2}
\sum_{k=1}^{\infty} (2k-1)\exp(-(2k-1)\beta)\left(\psi(k-1/4) +\psi(k+1/4)\right)\\ -
\frac{\cosh\beta}{4\pi^2\sinh^2\beta}\left(3\gamma+\frac{\pi}{2}-3\log(2\pi)\right)+
\Sigma(1/2)-\Sigma_B(1/2)\\-
\frac{2\zeta(1/2)}{\pi}\int_{0}^{\infty}
\frac{\zeta(1/2+2it)\zeta(1/2-2it)}{|\zeta(1+2it)|^2}
\hat{\varphi}(t)dt-2\frac{\zeta(0)}{\zeta(3/2)}\hat{\varphi}\left(\frac{1}{4i}\right),
\end{multline}
where $\Sigma(1/2)$ is given by \eqref{Sigma(1,s)} and $\Sigma_B(1/2)$ by \eqref{SigmaB(1,s)} with the omitted summand containing $\zeta(2s-1)$.
\end{lem}
\begin{proof}
The two summands
that have a simple pole at $s=1/2$ are
\begin{multline}\label{two poles}
\zeta(2s)\varphi_0-
\frac{1}{\pi^2}\frac{(2\pi)^s}{\sqrt{\pi}}\zeta(2s-1)\cos\left(\frac{\pi s}{2}\right)\Gamma(s-1/2)\times\\
\sum_{k=1}^{\infty} (2k-1)\exp(-(2k-1)\beta)
\frac{\Gamma(k-s/2)\Gamma(k+1/2-s/2)}{\Gamma(k+s/2)\Gamma(k-1/2+s/2)}.
\end{multline}
For simplicity let $s=1/2+u$ with $u\rightarrow0.$ Using \eqref{phi0 expression} and the functional equation for the Riemann zeta-function we obtain that \eqref{two poles} is equal to
\begin{multline}\label{two poles2}
\zeta(1+2u)\frac{-\cosh\beta}{2\pi^2\sinh^2\beta}-
\frac{\zeta(1-2u)}{\pi^2}\frac{(2\pi)^{1/2+3u}}{\sqrt{\pi}}
\cos\left(\frac{\pi}{4}+\frac{\pi u}{2}\right)\frac{\Gamma(1-2u)}{\Gamma(1-u)}\times\\
\sum_{k=1}^{\infty} (2k-1)\exp(-(2k-1)\beta)
\frac{\Gamma(k-1/4-u/2)\Gamma(k+1/4-u/2)}{\Gamma(k+1/4+u/2)\Gamma(k-1/4+u/2)}.
\end{multline}
Note that \eqref{two poles2} is holomorphic at $u=0$. This follows from the relation
\begin{equation*}
\sum_{k=1}^{\infty} (2k-1)\exp(-(2k-1)\beta)=\frac{\cosh\beta}{2\sinh^2\beta}.
\end{equation*}
Computing the limit of \eqref{two poles2} as $u$ tends to $0$ by L'hopital's rule we prove the lemma.
\end{proof}

\begin{thm} Let $X>5$ then for $s=1/2+ir$ the following holds
\begin{multline}\label{symsquare estimate}
M_1(s)\ll(1+|r|)^AT\log^3T+\\
\frac{(1+|r|)^AX^{\theta}}{X^{1/4}}\left(
X^{1/2}\min\left(T,\frac{X^{1/2}}{\fr(X)}\right)^{1/2}+
\min\left(T,\frac{X^{1/2}}{\fr(X)}\right)^{3/2}\right).
\end{multline}
\end{thm}
\begin{proof}
This is a consequence of    \eqref{SigmaB(s) estmate}, \eqref{Sigma(s) estmate for big T}, \eqref{Sigma(s) estmate} and \eqref{firstmoment=MT+Sigma-SigmaB-ContSpectr Res=1/2}. To estimate the integral containing Riemann zeta functions in \eqref{firstmoment=MT+Sigma-SigmaB-ContSpectr Res=1/2} it is sufficient to use the standard bound $|\zeta^{-1}(1+it)|\ll \log t,$ Cauchy-Schwarz inequality and the estimate on the second moment of the Riemann-zeta function
\begin{equation*}
\int_0^T|\zeta(1/2+it)|^2dt\ll T\log T.
\end{equation*}
As a result,
\begin{equation*}
\frac{2\zeta(s)}{\pi}\int_{0}^{\infty}
\frac{\zeta(s+2it)\zeta(s-2it)}{|\zeta(1+2it)|^2}
\hat{\varphi}(t)dt\ll (1+|r|)^AT\log^3T.
\end{equation*}
From \eqref{phiHat-Def} it follows that
\begin{equation*}
2\frac{\zeta(2s-1)}{\zeta(2-s)}\hat{\varphi}\left(\frac{1-s}{2i}\right)\ll
(1+|r|)^AX^{1/4}.
\end{equation*}
This error term is absorbed by other terms in \eqref{symsquare estimate}.
\end{proof}

\section{Estimates for the spectral exponential sum}\label{sec:estimatessum}

\begin{thm}\label{thm:4estimates} For $X>5$ one has
\begin{equation}\label{spec.sum estimate1}
\sum_{j}X^{it_j}\exp(-t_j/T)\ll\\ S(X,T,\fr(X))\log^{2}T,
\end{equation}
where
\begin{equation}\label{S(X,T,fr(X))def1}
S(X,T,\fr(X))=
\left\{
                \begin{array}{ll}
                  X^{1/4+\theta/2}T^{1/2} & \hbox{if}\quad T\le X^{1/2}; \\
                  X^{\theta/2}T & \hbox{if}\quad X^{1/2}<T\le \frac{X^{1/2}}{\fr(X)}; \\
                  X^{\theta/2}T^{1/4}\left(\frac{X^{1/2}}{\fr(X)}\right)^{3/4} & \hbox{if}\quad \frac{X^{1/2}}{\fr(X)}<T\le \frac{X^{1/2+2\theta/3}}{\fr(X)}; \\
                  T & \hbox{if}\quad \frac{X^{1/2+2\theta/3}}{\fr(X)}<T.
                \end{array}
              \right.
\end{equation}
\end{thm}
\begin{proof}
Substituting \eqref{symsquare estimate} to \eqref{spec.sum decomposition10} we obtain
\begin{multline*}\label{spec.sum decomposition4}
\sum_{j}X^{it_j}\exp(-t_j/T)\ll
\frac{NX^{1/2}\log T}{(\max(1,NX^{1/2}/T))^{1/2}}+
\sqrt{N/X}\log^2N+T\log^2T+\\
\frac{1}{N^{1/2}}\Biggl(
T\log^3T+
\frac{X^{\theta}}{X^{1/4}}\left(
X^{1/2}\min\left(T,\frac{X^{1/2}}{\fr(X)}\right)^{1/2}+
\min\left(T,\frac{X^{1/2}}{\fr(X)}\right)^{3/2}\right)
\Biggr).
\end{multline*}
Now we consider different cases.

 If $T\le X^{1/2}$ then taking $N=X^{\theta}\log^{-1}T$ we obtain
\begin{equation}\label{spec.sum estimate T<X^{1/2}}
\sum_{j}X^{it_j}\exp(-t_j/T)\ll X^{1/4+\theta/2}T^{1/2}\log^{1/2}T.
\end{equation}

 If $X^{1/2}<T\le X^{1/2}/\fr(X) $ then choosing $N=TX^{\theta-1/2}\log^{-1}T$ we prove that
\begin{equation}\label{spec.sum estimate X^{1/2}<T<X^{1/2}/fr}
\sum_{j}X^{it_j}\exp(-t_j/T)\ll X^{\theta/2}T\log^{1/2}T.
\end{equation}

 If $X^{1/2}/\fr(X)<T\le X^{1/2+2\theta/3}/\fr(X) $ then letting
\begin{equation*}
N=\frac{X^{\theta}}{T^{1/2}X^{1/2}\log T}\left(\frac{X^{1/2}}{\fr(X)}\right)^{3/2}
\end{equation*}
it follows that
\begin{multline}\label{spec.sum estimate X^{1/2}/fr<T<X^{1/2+2theta/3}/fr}
\sum_{j}X^{it_j}\exp(-t_j/T)\ll X^{\theta/2}T^{1/4}\left(\frac{X^{1/2}}{\fr(X)}\right)^{3/4}\log^{1/2}T+T\log^{2}T\ll\\
X^{\theta/2}T^{1/4}\left(\frac{X^{1/2}}{\fr(X)}\right)^{3/4}\log^{2}T.
\end{multline}

 If $X^{1/2+2\theta/3}/\fr(X)<T$ then taking $N=X^{2\theta/3}/\fr(X)$ we have
\begin{equation}\label{spec.sum estimate X^{1/2+2theta/3}/fr<T}
\sum_{j}X^{it_j}\exp(-t_j/T)\ll T\log^{2}T+\frac{X^{1/2+2\theta/3}}{\fr(X)}\log T\ll T\log^{2}T.
\end{equation}

\end{proof}
Note that it is sufficient to use only the first estimate in \eqref{S(X,T,fr(X))def1} in order to prove \eqref{PrimeGeodesic bound}.
Other estimates in \eqref{S(X,T,fr(X))def1}, as well as Lemma \ref{lemma Sigma estimate2}, serve to establish \eqref{PetrisdisRisager conj} unconditionally in some ranges. 

\begin{thm} For $X>5$ the following estimates hold
\begin{equation}\label{spec.sum.unsmoothed estimate1}
\sum_{t_j\le T}X^{it_j}\ll \max\left(
X^{1/4+\theta/2}T^{1/2},
X^{\theta/2}T
\right)\log^{3}T,
\end{equation}
\begin{equation}\label{spec.sum.unsmoothed estimate2}
\sum_{t_j\le T}X^{it_j}\ll T\log^{2}T,\quad\hbox{if}\quad T>\frac{X^{1/2+7\theta/6}}{\fr(X)}.
\end{equation}
\end{thm}
\begin{proof}
To prove the theorem we need to substitute the bound \eqref{spec.sum estimate1} to equation \eqref{spec.sum to spec.sum.exp}. Consequently,
\begin{equation}\label{spec.sum.unsmoothed estimate by integral}
\sum_{t_j\le T}X^{it_j}\ll \log^{2}T
\int_{-1}^1|\hat{g}(\xi)|S(X\exp(-2\pi\xi),T,\fr(X\exp(-2\pi\xi)))d\xi.
\end{equation}
Applying the estimate
\begin{equation}\label{S(X,T,fr(X))estimate}
S(X,T,\fr(X))\le\max(X^{1/4+\theta/2}T^{1/2},X^{\theta/2}T)
\end{equation}
that follows from \eqref{S(X,T,fr(X))def1}, we immediately obtain \eqref{spec.sum.unsmoothed estimate1}.
Estimate \eqref{spec.sum.unsmoothed estimate2} can be proved using the fact that
\begin{equation}\label{S(X,T,fr(X))estimate2}
S(X,T,\fr(X))=T\quad \hbox{if}\quad X^{1/2+2\theta/3}/\fr(X)<T.
\end{equation}
To this end, we decompose the integral \eqref{spec.sum.unsmoothed estimate by integral} into two parts. The first one is over $\xi$ such that $\fr(X\exp(-2\pi\xi))\le \delta.$ To estimate this integral we will apply \eqref{S(X,T,fr(X))estimate}. To estimate the second integral we will use \eqref{S(X,T,fr(X))estimate2}. According to these estimates, we will choose an optimal value of parameter $\delta$.

First, making the change of variables we have
\begin{equation}\label{spec.sum.unsmoothed estimate by integral2}
\sum_{t_j\le T}X^{it_j}\ll\frac{\log^{2}T}{X}
\int_{X\exp(-2\pi)}^{X\exp(2\pi)}\left|\hat{g}\left(\frac{\log(Z/X)}{2\pi}\right)\right|S(Z,T,\fr(Z))dZ.
\end{equation}
Denote  $Z_n=(n+\sqrt{n^2-4})^2/4$ and let
\begin{equation*}
Z_{n_1-1}<X\exp(-2\pi)\le Z_{n_1},\quad Z_{n_2}\le X\exp(2\pi)<Z_{n_2+1}.
\end{equation*}
Consequently,
\begin{multline*}
\sum_{t_j\le T}X^{it_j}\ll
\frac{\log^{2}T}{X}\int_I\left|\hat{g}\left(\frac{\log(Z/X)}{2\pi}\right)\right|S(Z,T,\fr(Z))dZ
\\+\frac{\log^{2}T}{X}
\sum_{n_1\le n\le n_2}
\int_{|Z-Z_n|\le\delta\sqrt{Z_n}}\left|\hat{g}\left(\frac{\log(Z/X)}{2\pi}\right)\right|S(Z,T,\fr(Z))dZ,
\end{multline*}
where
\begin{equation*}
I=[X\exp(-2\pi),X\exp(2\pi)]\backslash\bigcup_{n_1\le n\le n_2}[Z_n-\delta\sqrt{Z_n},Z_n+\delta\sqrt{Z_n}].
\end{equation*}
For $Z\in I$ one has $\fr(Z)\gg\delta.$ To estimate the integral over $I$ we use \eqref{S(X,T,fr(X))estimate2}. Bounds for the remaining integrals rely on \eqref{S(X,T,fr(X))estimate}.  Finally, we obtain for $T> X^{1/2+2\theta/3}/\delta$
\begin{multline}\label{spec.sum.unsmoothed estimate by integral3}
\sum_{t_j\le T}X^{it_j}\ll T\log^{3}T+\\
\frac{TX^{\theta/2}\log^{2}T}{X}\sum_{n_1\le n\le n_2}
\int_{|Z-Z_n|\le\delta\sqrt{Z_n}}\left|\hat{g}\left(\frac{\log(Z/X)}{2\pi}\right)\right|dZ.
\end{multline}
It follows from \eqref{g(x) estimate} that
\begin{equation*}\label{interval with big g}
\left|\hat{g}\left(\frac{\log(Z/X)}{2\pi}\right)\right|\ll T \quad\hbox{when}\quad |Z-X|\ll\frac{X}{T}.
\end{equation*}
Note  that if
\begin{equation}\label{conditions on fr(X)}
\fr(X)\gg\delta\quad\hbox{and}\quad \fr(X)X^{1/2}\gg\frac{X}{T},
\end{equation}
then the interval $|Z-X|\ll\frac{X}{T}$  does not intersect with
\begin{equation*}
\bigcup_{n_1\le n\le n_2}[Z_n-\delta\sqrt{Z_n},Z_n+\delta\sqrt{Z_n}].
\end{equation*}
Thus applying \eqref{g(x) estimate} we obtain
\begin{equation}\label{estimate on nth integral}
\int_{|Z-Z_n|\le\delta\sqrt{Z_n}}\left|\hat{g}\left(\frac{\log(Z/X)}{2\pi}\right)\right|dZ\ll
X\int_{|X\exp(y)-Z_n|\le\delta\sqrt{Z_n}}\frac{dy}{y}.
\end{equation}
Let $Z_{n_0}$ be the nearest $Z_n$ to $X$ and let $Y_j=Z_{n_0+j}.$  Then
\begin{multline}\label{sum over n interval to sum over j interval}
\sum_{n_1\le n\le n_2}
\int_{|Z-Z_n|\le\delta\sqrt{Z_n}}\left|\hat{g}\left(\frac{\log(Z/X)}{2\pi}\right)\right|dZ \\ \ll
\sum_{|j|\ll X^{1/2}}
X\int_{|X\exp(y)-Y_j|\le\delta\sqrt{Y_j}}\frac{dy}{y}.
\end{multline}
The interval of integration in \eqref{sum over n interval to sum over j interval} is equal to
\begin{equation}\label{jth interval}
\log\left(\frac{Y_j}{X}-\frac{\delta\sqrt{Y_j}}{X}\right)\le y\le
\log\left(\frac{Y_j}{X}+\frac{\delta\sqrt{Y_j}}{X}\right).
\end{equation}
Since  the point $Y_j$ is very close to the point
$X\pm\fr(X)\sqrt{X}+j\sqrt{X}$,  the interval has the following form
\begin{equation}\label{jth interval2}
\log\left(1+\frac{j\pm\fr(X)-\delta}{X^{1/2}}\right)\le y\le
\log\left(1+\frac{j\pm\fr(X)+\delta}{X^{1/2}}\right).
\end{equation}
Using the fact that for  a small $\alpha$
\begin{equation*}
\int_{\log a}^{\log(a+\alpha)}\frac{dy}{y}\ll\frac{\alpha}{\log a}
\end{equation*}
and applying estimate \eqref{conditions on fr(X)}, we prove
\begin{equation}\label{estimate on jth intagral}
\int_{|X\exp(y)-Y_j|\le\delta\sqrt{Y_j}}\frac{dy}{y}\ll\frac{\delta}{|j+\fr(X)|}.
\end{equation}
Substituting \eqref{estimate on jth intagral} to \eqref{sum over n interval to sum over j interval} yields
\begin{equation}\label{sum over n interval estimate}
\sum_{n_1\le n\le n_2}
\int_{|Z-Z_n|\le\delta\sqrt{Z_n}}\left|\hat{g}\left(\frac{\log(Z/X)}{2\pi}\right)\right|dZ\ll
\frac{X\delta}{\fr(X)}\log X.
\end{equation}
Substituting \eqref{sum over n interval estimate} to \eqref{spec.sum.unsmoothed estimate by integral3} we obtain
\begin{equation}
\sum_{t_j\le T}X^{it_j}\ll T\log^{3}T+
\frac{TX^{\theta/2}\log^{3}T}{\fr(X)}\delta
\end{equation}
under conditions \eqref{conditions on fr(X)} and for $T> X^{1/2+2\theta/3}/\delta$. Choosing $\delta=\fr(X)X^{-\theta/2}$, we conclude that  \eqref{conditions on fr(X)} is satisfied under assumptions of the lemma. This completes the proof of \eqref{spec.sum.unsmoothed estimate2}.
\end{proof}
\section*{Acknowledgements}
We thank the referee for many helpful suggestions.


\nocite{}


\begin{thebibliography}{}


\bibitem{BF1}
\newblock
O. Balkanova and D. Frolenkov, \emph{On the mean value of symmetric square
$L$-functions},  Algebra Number Theory  12 (2018), 35--59.

\bibitem{BF2}
\newblock
O. Balkanova and D. Frolenkov, \emph{Convolution formula for the sums of generalized Dirichlet $L$-functions}, Revista Math. Iberoamericana, to appear, arXiv:1709.01365 [math.NT].

\bibitem{B}
\newblock
V.A. Bykovskii, \emph{Density theorems and the mean value of arithmetic functions on short intervals}. (Russian) Zap. Nauchn. Sem. S.-Peterburg. Otdel. Mat. Inst. Steklov. (POMI) 212 (1994), Anal. Teor. Chisel
i Teor. Funktsii. 12, 56--70, 196; translation in J. Math. Sci. (New York) 83 (1997), no. 6, 720--730.

\bibitem{BykF}
\newblock
V. A. Bykovskii and D. A. Frolenkov, \emph{Asymptotic formulas for the second moments of $L$-series associated to holomorphic cusp forms on the critical line},   Izvestiya: Mathematics 81:2, 2017,  239--268.

\bibitem{BE}
\newblock
H. Beitman and A. Erdelyi, \emph{Tables of integral transforms}, Vol. 1, McGraw-Hill, New York, 1954.

\bibitem{Cai}
\newblock
Y. Cai, \emph{Prime geodesic theorem}, J. Theor.  Nombres Bordeaux 14:1 (2002), 59--72.

\bibitem{CG}
G. Cherubini and J. Guerreiro, \emph{Mean square in the prime geodesic theorem}, 	Algebra Number Theory 12 (2018), 571--597.


\bibitem{CI}
\newblock
J. B. Conrey and H. Iwaniec, \emph{The cubic moment of central values of automorphic
$L$-functions}, Ann. of Math. (2) 151 (2000), 1175--1216.

\bibitem{DesIw}
\newblock
J. M. Deshouillers and H. Iwaniec, \emph{The non-vanishing of Rankin-Selberg zeta-functions at special points}, Selberg trace formula and related topics, Contemp.  Math. 53, Amer. Math. Soc., Providence, RI, 1986, 59--95.

\bibitem{Frol}
\newblock
D. A. Frolenkov,  \emph{On the uniform bounds on hypergeometric function},
Far Eastern Math. J. 15:2 (2015), 288--298.

\bibitem{GR}
\newblock
I. S. Gradshteyn and I. M. Ryzhik, \emph{ Table of Integrals, Series, and Products}. Edited by A. Jeffrey and D. Zwillinger. Academic Press, New York, 7th edition, 2007.



\bibitem{IwPG}
\newblock
 H. Iwaniec, \emph{Prime geodesic theorem}, J. Reine. Angew. Math. 349 (1984), 136--159.


\bibitem{Iwbook}
\newblock
 H. Iwaniec, \emph{Introduction to the spectral theory of automorphic forms}, Revista Matem\'{a}tica Iberoamericana, 1995.


\bibitem{Kuz}
\newblock
 N. V. Kuznetsov,  \emph{Petersson's conjecture for cusp of weight zero and Linnik's conjecture, Sums of Kloosterman sums}, Mat. Sb. 111 (1980), 334--383.


\bibitem{Ler}
\newblock
M. Lerch, \emph{Note sur la fonction $R(w,x,s)=\sum_{k=0}^{\infty}\frac{e^{2\pi ina}}{(n+c)^s}$}, Acta Math. 11 (1887), 19--24.

\bibitem{Luo}
\newblock
 W. Luo, \emph{Values of symmetric square L-functions at 1}, J. Reine Angew. Math. 506 (1999), 215--235.
 
\bibitem{LuoSarnakPG}
\newblock
 W. Luo and P. Sarnak, \emph{Quantum ergodicity of eigenfunctions on $PSL_2(Z)/H^2$}, Pub. math. de l'I.H.E.S. 81 (1995), 207--237.

 \bibitem{HMF}
F. W. J.~Olver , D. W.~Lozier, R. F.~Boisvert and C. W.~Clarke, \emph{{NIST} {H}andbook of {M}athematical {F}unctions}, Cambridge University Press, Cambridge $(2010)$.

 \bibitem{PetRis}
\newblock
Y. N. Petridis and M. S. Risager, \emph{ Averaging over Heegner points in the hyperbolic circle problem}, IMRN (2017), https://doi.org/10.1093/imrn/rnx026.

\bibitem{PetRisLaak}
\newblock
Y. N. Petridis and M. S. Risager, \emph{Local average in hyperbolic lattice point counting, with an appendix by N. Laaksonen},  Math. Z. 285 (2017), no. 3-4, 1319--1344.

\bibitem{SY}
K. Soundararajan and M. P. Young, \emph{The prime geodesic theorem}, J. Reine Angew. Math. 676 (2013), 105--120.

\bibitem{W}
\newblock
A. Weil, \emph{On some exponential sums}, Proc. Natl. Acad. Sci. USA 34 (1948),  204--207.

\bibitem{Y}
\newblock
M.P. Young, \emph{Weyl-type hybrid subconvexity bounds for twisted $L$-functions and Heegner points on shrinking sets}, J. Eur. Math. Soc. 19 (2017), 1545--1576. 

\bibitem{Z}
\newblock
D. Zagier, \emph{ Modular forms whose Fourier coefficients involve zeta-functions of quadratic fields}. Modular
functions of one variable, VI (Proc. Second Internat. Conf., Univ. Bonn, Bonn, 1976), pp. 105--169.
Lecture Notes in Math., Vol. 627, Springer, Berlin, 1977.




\end{thebibliography}
\end{document}